\documentclass[10pt]{article}
\usepackage[bookmarks]{hyperref}
\usepackage{amssymb,amsmath,amsthm,amscd,verbatim}
\usepackage{amsrefs}
\usepackage{color}
\usepackage{graphicx}
\newtheorem{thm}{Theorem}[section]
\newtheorem{lem}[thm]{Lemma}

\theoremstyle{definition}
\newtheorem{df}[thm]{Definition}

\newtheorem{que}[thm]{Question}

\theoremstyle{remark}

\newtheorem{cor}[thm]{Corollary}

\newcommand{\la}{\langle}
\newcommand{\ra}{\rangle}

\renewcommand{\iff}{\leftrightarrow}

\newcommand{\mc}{\mathcal}

\newcommand{\union}{\cup}

\newcommand{\inter}{\cap}
\newcommand{\Union}{\bigcup}

\renewcommand{\to}{\rightarrow}

\newcommand{\restrict}{\upharpoonright}

\newcommand{\nil}{\varnothing}
\renewcommand{\P}{\mathbb P}
\newcommand{\E}{\mathbb E}

\DeclareMathOperator{\DNR}{DNR}

\begin{document}
\title{A strong law of computationally weak subsets\footnote{Electronic version of an article published as
	\emph{J. Math. Log.} \textbf{11}, Issue 1 (2011), 1--10. DOI: 10.1142/S0219061311000980.
	\copyright World Scientific Publishing Company. Journal URL: \href{http://www.worldscientific.com/worldscinet/jml}{http://www.worldscientific.com/worldscinet/jml}.}}
\author{Bj\o rn Kjos-Hanssen\footnote{This material is based upon work supported by the National Science Foundation under Grants No. 0652669 and 0901020. The author is grateful for an excellent referee report on the paper.}\\ University of Hawai\textquoteleft i at M\=anoa }
	
\maketitle

\begin{abstract}
We show that in the setting of fair-coin measure on the power set of the natural numbers, each sufficiently random set has an infinite subset that computes no random set. That is, there is an almost sure event $\mc A$ such that if $X\in\mc A$ then $X$ has an infinite subset $Y$ such that no element of $\mc A$ is Turing computable from $Y$.
\end{abstract}

\section{Introduction} 
The practical utility of random bits being well established, we view randomness of an infinite set of positive integers $X$ as a valuable property. The question arises whether if a set $Y$ is close to $X$ in some sense, then $Y$ retains some of the value of $X$ in that, even if $Y$ is not itself random, one can compute a random sequence from $Y$. 

There are various ways in which $Y$ and $X$ could be considered ``close''; a natural one is to assume $Y\subseteq X$ and $Y$ is infinite.\footnote{This is not a natural notion of closeness for subsets of the plane, say, but rather in terms of the information provided: if $Y$ is a subset of $X$, i.e.\ whenever $n\in Y$, $n\in X$, then $Y$ provides some reliable information about $X$.} In this article we shall prove that under the fair-coin Lebesgue measure there is an almost sure event $\mc A$ such that if $X\in\mc A$ then $X$ has an infinite subset $Y$ such that no element of $\mc A$ is Turing reducible to $Y$. This confirms the intuition one may have that a subset of a random set should not generally be able to compute a random set. This ``strong law of computationally weak subsets'' is a probabilistic law in the same sense as the strong law of large numbers; it gives an almost sure property. A key to the proof will be the classical extinction criterion for Galton-Watson processes, Theorem \ref{extinctionCriterion}.

Our results improve upon an earlier result \cite{K:MRL} to the effect that there simply exists a Martin-L\"of random set $X$ and an infinite subset $Y$ of $X$ such that no Martin-L\"of random set is Turing reducible to $Y$. A different proof of that result may be deduced from work of Greenberg and Miller \cite{GM}. 

It is worthwhile to note that the computationally weak subset $Y$ cannot be chosen too sparse or too dense. If $Y$ is very sparse then the intervals between subsequent elements of $Y$ are longer than the running time (when finite) of a universal Turing machine, and so $Y$ solves the halting problem. If $Y$ is very dense then a majority function applied to disjoint intervals of bits of $Y$ will produce a random set. Another sense in which the construction of $Y$ must be nontrivial is that if $Y=X\cap R$, where $R$ is an infinite computable set, then $Y$ computes a random set. (Namely, the image of $Y$ under a computable bijection from $R$ to the natural numbers.) Instead, $Y$ will be obtained as the image of an infinite path through a certain noncomputable tree; $Y$ will exist by the extinction criterion for Galton-Watson processes. 

\section{Bushy trees}

Our overall plan is to apply the extinction criterion to a \emph{bushy} infinite tree, each path through which obeys a construction like that of Ambos-Spies, Kjos-Hanssen, Lempp, and Slaman \cite{AKLS}. Let $\omega=\{0,1,2,\ldots\}$ be the set of natural numbers and let $\omega^{<\omega}$ be the set of finite strings over $\omega$. If $\sigma,\tau\in\omega^{<\omega}$ then $\sigma$ is called a \emph{substring} of $\tau$, $\sigma\subseteq\tau$, if for all~$x$ in the domain of~$\sigma$, $\sigma(x)=\tau(x)$.

\begin{df}[\cite{AKLS}, \cite{KL}]
A finite set of incomparable strings in $\omega^{<\omega}$ is called a \emph{leaf bag}.\footnote{In \cite{AKLS} a leaf bag was, slightly confusingly, called a \emph{tree}.} Given $n\in\omega$, a nonempty leaf bag~$T$ is called \emph{$n$-bushy from $\sigma\in\omega^{<\omega}$} if 
\begin{itemize}
\item[(1)] every string $\tau\in T$ extends~$\sigma$, and
\item[(2)] for each $\tau\in\omega^{<\omega}$, if there exists $\rho\in T$ with $\sigma\subseteq\tau\subset\rho$, then there are at least $n$ many immediate successors of~$\tau$ which are substrings of elements of~$T$.
\end{itemize}
If $T$ is $n$-bushy from~$\sigma$ and $T\subseteq P\subseteq\omega^{<\omega}$, then $T$ is called \emph{$n$-bushy from~$\sigma$ for~$P$}.
\end{df}

\begin{df}
A set $C\subseteq\omega^{<\omega}$ is \emph{$n$-perfectly bushy} if the empty string is in $C$ and every element of $C$ has at least $n$ many immediate extensions in $C$.
\end{df}
An $n$-perfectly bushy set $C$ is in particular a tree in the sense of being closed under substrings, and the set of infinite paths through $C$ is a perfect set  $[C]\subset\omega^{\omega}$.

\begin{lem}[an extension of \cite{AKLS}*{Lemma 2.5}]\label{kum-basic} 
Let $n\ge 1$. Given a leaf bag~$T$ that is $(a+b-1)$-bushy from a string~$\alpha$ and given a set $P\subseteq T$, there is a subset~$S$ of~$T$ which is $a$-bushy for~$P$ or $b$-bushy for $T-P$.
\end{lem}
\begin{proof}
Give the elements of~$T$ the label 1 (0) if they are in~$P$ (not in~$P$, respectively). Inductively, suppose $\beta$ extends $\alpha$ and is a proper substring of an element of~$T$. Suppose all the immediate successors of $\beta$ that are substrings of elements of~$T$ have received a label. Give $\beta$ the label 1 if at least $a$ many of its labelled immediate successors are labelled 1; otherwise, give $\beta$ the label 0. (In this case, at least $(a+b-1) - (a-1) = b$ many immediate successors are labelled $0$.) This process ends after finitely many steps when $\alpha$ is given some label $i\in\{0,1\}$. Let $S$ be the set of $i$-labelled strings in~$T$. If $i=1$ then $S$ is contained in~$P$, and if $i=0$ then $S$ is contained in~$T-P$, so it only remains to show that $S$ is $a\mathbf 1_{\{1\}}(i) + b\mathbf 1_{\{0\}}(i) = ai+b(1-i)$-bushy.\footnote{Here $\mathbf 1_{A}(n)=1$ if $n\in A$, and $\mathbf 1_{A}(n)=0$ otherwise.}

Let $L$ be the set of all labelled strings. Note that $L$ is the set of strings extending $\alpha$ that are substrings of elements of $T$. For any $\beta\in L-T$, let $k$ be the number of immediate successors of $\beta$ that are in $L$. Since $T$ is $(a+b-1)$-bushy, $k\ge a+b-1$. Let $p\le k$ be the number of immediate successors of $\beta$ that have the same label as $\beta$. By construction, $p\ge ai+b(1-i)$. It follows that $S$ is $ai+b(1-i)$-bushy.
\end{proof}

\begin{lem}\label{notsharp-basic}
Let $a,n\ge 1$. Let $T$ be a leaf bag which is $2^{a-1}n$-bushy from a string
$\alpha$, and let $P_1,\ldots,P_a$ be sets of strings such that
$T\subseteq\Union_i P_i$. Then for some~$i$, $T$~has a subset which is
$n$-bushy from $\alpha$ for~$P_i$.
\end{lem}

\begin{proof}
The case $a=1$ is trivial; the subset is $T$ itself. So assume $a\ge 2$ and
assume that Lemma~\ref{notsharp-basic} holds with $a-1$ in place of~$a$. By
Lemma~\ref{kum-basic}, if there is no $2^{a-2}n$-bushy subset of~$T$ from
$\alpha$ for~$P_1$ then there is a $2^{a-2}n$-bushy subset~$S$ of~$T$ from
$\alpha$ for the complement $\overline P_1$. As $T\inter\overline
P_1\subseteq P_2\union\cdots\union P_a$, it follows that $S$ is
$2^{a-2}n$-bushy from~$\alpha$ for $P_2\union\cdots\union P_a$. By Lemma
\ref{notsharp-basic} with $a-1$ in place of~$a$, $S$ has a subset~$R$ which
is $n$-bushy from~$\alpha$ for some~$P_i$, $i\ge 2$. As $R$ is also a subset
of~$T$, the proof is complete.
\end{proof}

We now need a simple but crucial strengthening of \cite{AKLS}*{Lemma 2.10}; the difference is that nonemptiness is replaced by bushiness. 

\begin{lem}\label{induct-basic}
Let $\Delta\in\omega$. Suppose we are given $\alpha$ and~$n$ and a set $P\subseteq \omega^{<\omega}$ such that there is no $n$-bushy leaf bag from~$\alpha$ for~$P$. If $V$ is an $n+\Delta-1$-bushy leaf bag from~$\alpha$ then there exists a $\Delta$-bushy set of strings $T$ such that for each $\beta\in T$,
	\begin{enumerate}
		\item $\beta\in V$, and 
		\item there is no $n$-bushy leaf bag from~$\beta$ for~$P$.
	\end{enumerate}
\end{lem}

\begin{proof}
Fix $V$ and suppose there is no such set $T$. By Lemma \ref{kum-basic} there is an $n$-bushy set $B\subseteq V$ such that for all $\beta\in B$, there \emph{is} an $n$-bushy leaf bag $V_{\beta}$ from $\beta$ for $P$; then
\[
	V^*=\Union\limits_{\beta \supseteq \alpha,\,\beta\in B} V_\beta
\]
would be $n$-bushy from~$\alpha$ for~$P$.
\end{proof}

\section{Diagonalization}

Diagonally non-recursive functions will be our bridge between randomness and bushy trees. To a certain extent this section follows Ambos-Spies et al.\ \cite{AKLS}.

\begin{df}\label{firstdef}
The length of a string~$\sigma$ is denoted by~$|\sigma|$. A string $\la a_1,\ldots,a_n\ra\in\omega^n$ is denoted $(a_1,\ldots,a_n)$ when we find this more natural. 
The \emph{concatenation} of $\la a_1,\ldots,a_n\ra$ by $\la a_n\ra$ on the right is denoted $\la a_1,\ldots,a_n\ra*\la a_{n+1}\ra=\la a_1,\ldots,a_n\ra*a_{n+1}$. If $G\in\omega^\omega$ then $\sigma$ is a substring of~$G$ if for all~$x$ in the domain of~$\sigma$, $\sigma(x)=G(x)$.

Let $\Phi_n$, $n\in\omega$, be a standard list of the Turing functionals. So if $A$ is recursive in~$B$ then for some $n$, $A=\Phi_n^B$. For convenience, if $\Phi$ is a Turing functional and for all $B$ and $x$, the computation of $\Phi^B(x)$ is independent of~$x$, we sometimes write $\Phi^B$ instead of $\Phi^B(x)$. Let $\Phi_{n,t}$ be the modification of $\Phi_n$ which goes into an infinite loop after~$t$ computation steps if the computation has not ended after~$t$ steps. We abbreviate $\Phi_n^\emptyset$ by $\Phi_n$. If the computation $\Phi_e(x)$ terminates we write $\Phi_e(x)\downarrow$, otherwise $\Phi_e(x)\uparrow$.
\end{df}

\begin{df}
Given functions $H,G:\omega\to\omega$, we say $H$ is $\DNR$ (\emph{diagonally nonrecursive}) if for all $x\in\omega$, $H(x)\ne\Phi^G_x(x)$ or $\Phi^G_x(x)\uparrow$. Given
$h:\omega\to\omega$, we say $H$ is $h$-$\DNR$ if in addition for all~$n$, $H(n)<h(n)$. (This necessitates that $h(n)>0$ for all~$n$.) If $H$ is $\DNR$ and $\sigma$ is a substring of~$H$ then $\sigma$ is called a $\DNR$ \emph{string}.
\end{df}

\begin{df}
Let $F=\text{Fix}$ be a computable function such that for all $a\in\omega$, $\text{Fix}(a)$ is the fixed-point of $\Phi_a$ produced by the Recursion Theorem; thus, if $e=\text{Fix}(a)$ then 
\[
	\Phi_{e}(x)=\Phi_{\Phi_a(e)}(x)
\]
for all $x\in\omega$.
\end{df}

Throughout the rest of this article, fix a recursive function
$h:\omega\to\omega$ satisfying Theorem \ref{kurt}; for example, $h(n)=n^{2}$ works. 

\begin{df}\label{f-basic}
Given a string $\alpha\in\omega^{<\omega}$, $c\in\omega$, and $n\in\omega$, let $f=f_{\alpha,c,n}=\Phi_{\text{Search}(\alpha,c,n)}$ be defined by the condition: 
\begin{quote}
$\Phi_{\Phi_{\text{Search}(\alpha,c,n)}(e)}(x)=i$ if \\
there is a leaf bag~$T$ and a number $i<h(e)$ such that $T$ is $n$-bushy from~$\alpha$ for $\{\beta:\Phi_c^{\beta}(e)=i\}$ (and $i$ is the $i$ occurring for the first such leaf bag found). If such $T$ and $i$ do not exist then $\Phi_{f(e)}(x)\uparrow$.
\end{quote}

If we let $e=\text{Fix}(\text{Search}(\alpha,c,n))$ then consequently
\begin{quote}
$\Phi_e(x)=i$ if \\
there is a finite leaf bag~$T$ and a number $i<h(e)$ are found such that $T$ is $n$-bushy from~$\alpha$ for $\{\beta:\Phi_c^{\beta}(e)=i\}$ (and $i$ is the $i$ occurring for the first such leaf bag found). 
\end{quote}
\end{df}

\begin{df}\label{gvec-basic}
Let $\epsilon:\omega\to\omega$ be a finite partial function and write $e_t=\epsilon(t)$ for each~$t$ in the domain of~$\epsilon$.

Let $\Phi$ be any Turing functional such that for all $G:\omega\to\omega$,
\[
	\Phi^G(\epsilon)\downarrow\iff\exists t\in\text{dom}(\epsilon)\,[\Phi^G_t(e_t)\downarrow<h(e_t)].
\] 
Given $n\in\omega$ and~$\epsilon$, let $g(n,\epsilon)=2^a n$ where
\[
	a=\sum_{t\in \text{dom}(\epsilon)} h(e_t).
\]
\end{df}

\begin{lem}[\cite{AKLS}*{[Lemma 2.8}]\label{g-basic}
Let $n\ge 1$, let $\epsilon$ be a finite partial function from~$\omega$ to $\omega$, and let $g$ be the function defined in Definition~\ref{gvec-basic}. For each pair $(t,i)$ satisfying $i<h(e_t)$ and $t\in \mathop{\rm dom}(\epsilon)$, let $Q_{(t,i)}=\{\beta:\Phi^\beta_t(e_t)=i\}$. Let $Q=\{\beta:\Phi^\beta(\epsilon)\downarrow\}$. If there is a $g(n,\epsilon)$-bushy leaf bag for~$Q$ from some string~$\alpha$, then for some $(t,i)$, there is an $n$-bushy leaf bag from~$\alpha$ for $Q_{(t,i)}$.
\end{lem}
\begin{proof}
The number of pairs $(t,i)$ such that $Q_{(t,i)}$ is defined is
$$
a=\sum_{t\in \text{dom}(\epsilon)} h(e_t).
$$
By the assumption that there is a $g(n,\epsilon)$-bushy leaf bag for~$Q$, it
follows that $a>0$. So since $2^an\ge 2^{a-1}n$, every $2^a n$-bushy leaf bag is
$2^{a-1}n$-bushy. Now apply Lemma~\ref{notsharp-basic} to the properties
$Q_{(t,i)}$.
\end{proof}

If $C\subseteq\omega^{<\omega}$ and $G\in\omega^{\omega}$ then we say $G\in [C]$ if for all $n$, $G\restrict n\in C$. Let $0'$ denote the halting problem for Turing machines.

\begin{thm}\label{theone-basic}
Let $\Delta\in\omega$. There is a $\Delta$-perfectly bushy set $C\subseteq \omega^{<\omega}$, $C\le_{T}0'$, such that for each $G\in [C]$ and all Turing functionals~$\Phi$, $\Phi^G$ is not $h$-$\DNR$. 
\end{thm}

Towards proving Theorem \ref{theone-basic}, we use the following construction.

\begin{df} \emph{The Construction.}\label{const}
The construction depends on a parameter $\Delta\in\omega$.
At any stage $s+1$, the finite set $D_{s+1}$ will consist of indices $t\le s$ for computations $\Phi^G_t$ that we want to ensure to be divergent. The set $A_{s+1}$ will consist of what we think of as acceptable strings. The numbers $n[s]$ and $n[s+\frac{1}{2}]$ will measure the amount of bushiness required.

\noindent\emph{Stage 0.}

Let $G[0]=\emptyset$, the empty string, and $\epsilon[0]=\emptyset$. Let $n[0]=2$. Let $D_0=\emptyset$ and $A_0=\omega^{<\omega}$.

\noindent\emph{Stage $s+1$, $s\ge 0$.}

Let $n[s+\frac12]=g(n[s],\epsilon[s])$, with $g$ as in Definition~\ref{gvec-basic}. Let $n[s+1]=n[s+\frac12]+\Delta-1$. Below we will define $D_{s+1}$. Given $D_{s+1}$, $A_{s+1}$ will be 
\[
	A_{s+1}=\{\tau\supset G[s]\mid \neg(\exists t\in D_{s+1})(\exists i<h(e_{t})(\exists T)
\]
\[
	(\text{$T$ is a finite $n[s+1]$-bushy leaf bag from~$\tau$ for $Q_{(t,i)}$})\}
\]

Let $e$ be the fixed point of $f=f_{G[s],s,n[s+1]}$ (as defined in Definition \ref{f-basic}) produced by the Recursion Theorem, i.~e., $\Phi_e=\Phi_{f(e)}$.

\noindent\emph{Case 1.} $\Phi_e(e)\downarrow$.

Fix $T$ as in Definition~\ref{f-basic}. Let $D_{s+1}=D_s$. 
\[
	\tag{*}
	\text{Let $G[s+1]$
	 be an extension of~$G[s]$ with $G[s+1]\in T\cap A_{s+1}$.}
\]
\noindent\emph{Case 2.} $\Phi_e(e)\uparrow$. Let $D_{s+1}=D_s\union\{s\}$. Let $\epsilon[s+1]=\epsilon[s]\union\{(s,e)\}$. In other words, $e_s=\epsilon(s)$ exists and equals~$e$. 
\[
	\tag{+}
	\text{Let $G[s+1]$ be any element of~$A_{s+1}$.}
\]

\noindent Let $G=\Union_{s\in\omega} G[s]$.

\noindent\emph{End of Construction.}
\end{df}

\noindent
We now prove that the Construction satisfies Theorem~\ref{theone-basic} in a sequence of lemmas.

\begin{lem}\label{atleasttwo-basic}
For each $s,t\in\omega$ with $t\le s$, $n_t[s]\ge 2$.
\end{lem}

\begin{proof}
For $s=0$, we have $n[0]=2$. For~$s+1$, we have $n[s+1]=g(n[s],\epsilon[s])=2^a n[s]$ for a certain $a\ge 0$, by Definition \ref{g-basic}, hence the lemma follows.
\end{proof}

\begin{lem}\label{accept-basic}
For each $s\ge 0$ the following holds.

\begin{enumerate}
\item[(1)] The Construction at stage~$s$ is well-defined and $G[s]\in A_s$. In particular, if $s>0$ then in Case 2, $A_s$ is nonempty, and in Case 1, $A_s$ contains at least one element of~$T$.

\item[(2)] There is no $n[s+\frac12]$-bushy leaf bag for $Q=\{\beta:\Phi^\beta(\epsilon[s])\downarrow\}$ from~$G[s]$.

\item[(3)] Every leaf bag~$V$ which is $n[s+1]$-bushy from~$G[s]$, and is not just the singleton of~$G[s]$, contains a $\Delta$-bushy set of elements of~$A_{s+1}$.
\end{enumerate}
\end{lem}

\begin{proof}
It suffices to show that (1) holds for~$s=0$, and that for each $s\ge 0$, (1) implies (2) which implies (3), and moreover that (3) for~$s$ implies (1) for
$s+1$.

\noindent (1) holds for~$s=0$ because $G[0]=\emptyset\in\omega^{<\omega}=A_0$.

\noindent\emph{(1) implies (2):}

By definition of~$A_s$ and the fact that $G[s]\in A_s$ by (1) for~$s$, we have that for each $t\in D_s$, and each $i<h(e_t)$, there is no $n[s]$-bushy leaf bag from~$G[s]$ for $Q_{(t,i)}=\{\beta:\Phi^\beta_t(e_t)\downarrow=i\}$. Hence by Lemma~\ref{g-basic}, there is no $n[s+\frac12]$-bushy leaf bag for $Q=\{\beta:\Phi^\beta(\epsilon[s])\downarrow\}$ from~$G[s]$.

\noindent\emph{(2) implies (3):}

Since $V$ is $n[s+1]$-bushy, by Lemma~\ref{induct-basic} there is a $\Delta$-bushy set of elements $\beta$ of~$V$ from which there is no $n[s+\frac12]$-bushy leaf bag for~$Q$, and hence no $n[s+1]$-bushy leaf bag for $Q_{(t,i)}$ either, since $n[s+\frac12]\le n[s+1]$ and $Q_{(t,i)}\subseteq Q$. Moreover, each such $\beta$ properly extends~$G[s]$, since $V$ is an antichain and is not the singleton of~$G[s]$. Hence by definition of~$A_{s+1}$, each such element $\beta$ belongs to $A_{s+1}$.

\noindent\emph{(3) for~$s$ implies (1) for~$s+1$:}

If Case 1 holds, let $T$ be the leaf bag found by~$\Phi_e$, i.~e., $T$ is $n[s+1]$-bushy from~$G[s]$ (for~$Q_{(s,i)}$ for some~$i$). If $T$ is not just the singleton of~$G[s]$, and Case 1 holds, then apply (3) for~$s$ to~$T$.

If $T$ is just the singleton of~$G[s]$ or if Case 2 holds, then apply (3) for~$s$ to any $n[s+1]$-bushy non-singleton leaf bag from~$G[s]$. For example, this could be the set of immediate extensions~$G[s]*k$, $k<n[s+1]$.
\end{proof}

\begin{lem}
	\label{diverge-basic} 
	For any $s\ge 0$, if $s\in D_{s+1}$ then $\Phi_s^G(e_s)\uparrow$ or $\Phi_s^G(e_s)\ge h(e_s)$.
\end{lem}

\begin{proof}
Otherwise for some $t\in\omega$, $\Phi_s^{G[t]}(e_s)\downarrow<h(e_s)$. Since the singleton leaf bag $T=\{G[t]\}$ is $n$-bushy from~$G[t]$ for all~$n$, hence in particular $n[t]$-bushy, this contradicts the fact that by Lemma \ref{accept-basic}(1), $G[t]\in A_t$.
\end{proof}

\begin{lem}\label{total-basic}
$G$ is a total function, i.e., $G\in\omega^\omega$.
\end{lem}

\begin{proof}
By Lemma~\ref{accept-basic}(3), $G[s+1]\in A_{s+1}$ for each $s\ge 0$, and hence by definition of~$A_{s+1}$, $G[s+1]$ is a proper extension of~$G[s]$. From this the lemma immediately follows.
\end{proof}

\begin{lem}
$G$ computes no $h$-$\DNR$ function.
\end{lem}

\begin{proof}
If Case 1 of the construction is followed then $\Phi^G_s(e)=\Phi^{G[s+1]}(e)=\Phi_e(e)$ because $G[s+1]\in T$. So $\Phi_s^G$ is not $h$-$\DNR$. If Case 2 of the construction is followed then $s\in D_{s+1}$ and so $\Phi_s^G(e)\uparrow$ or $\Phi_s^G(e)\ge h(e)$ by Lemma \ref{diverge-basic}. Thus again $\Phi_s^G$ is not $h$-$\DNR$.
\end{proof}

\begin{proof}[Proof of Theorem \ref{theone-basic}] 
We showed how to construct a single $G\in\omega^{\omega}$, but since by Lemma \ref{accept-basic}(3) the choice of $G[s+1]$ can be made in a $\Delta$-bushy set of ways, the set $C$ of all functions $G$ obeying ($*$) and ($+$) in the construction \ref{const} is $\Delta$-perfectly bushy. Routine inspection show that the construction and hence the set $C$ are recursive in $0'$.
\end{proof}

\section{A law of weak subsets}

A sequence $X\in 2^{\omega}$ is also considered to be a set $X\subseteq\omega$.  For the notions of Martin-L\"of random and Schnorr random sets $X$ relative to an oracle $A$ we refer the reader to Nies' book \cite{NiesBook}. For $n\in\omega$, a set $X$ is $(n+1)$-random if it is Martin-L\"of random relative to the $n^{\text{th}}$ iteration of the Turing jump, $0^{(n)}$, and Schnorr $(n+1)$-random if it is Schnorr random relative to $0^{(n)}$. 

\begin{thm}[Ku\v{c}era~\cite{Kucera:84} and Kurtz  (see Jockusch \cite{Jockusch:89}*{Proposition 3}] \label{kurt}
There is a recursive function~$h$ such that for each Martin-L\"of random real
$R$, there is an $h$-$\DNR$ function~$f$ recursive in~$R$.
\end{thm}

Applying Theorem \ref{theone-basic}, we have
\begin{thm}\label{Main}
Let $\Delta\in\omega$. There is a $\Delta$-perfectly bushy set $C\subseteq \omega^{<\omega}$, $C\le_{T}0'$, such that for each $G\in [C]$ and each Martin-L\"of random set $X$, $X\not\le_{T} G$. 
\end{thm}

The key idea is now to consider the intersection of $C$ with a random set $X\subseteq\omega^{<\omega}$ as a Galton-Watson process. Theorem \ref{extinctionCriterion} can be considered the fundamental result in the theory of such processes. It was first stated by Bienaym\'e in 1845; see Heyde and Seneta \cite{HS}*{pp. 116--120} and Lyons and Peres \cite{LP}*{Proposition 5.4}. The first published proof appears in Cournot \cite{Cournot}*{pp. 83--86}. As usual, $\P$ denotes probability.

\begin{thm}[Extinction Criterion]\label{extinctionCriterion}
Given numbers $p_{k}\in [0,1]$ with $p_{1}\ne 1$ and $\sum_{k\ge 0}p_{k}=1$, let $Z_{0}=1$, let $L$ be a random variable with $\P(L=k)=p_{k}$, let $\{L^{(n)}_{i}\}_{n,i\ge 1}$ be independent copies of $L$, and let 
\[
Z_{n+1}=\sum_{i=1}^{Z_{n}}L_{i}^{(n+1)}.
\]
Let $q=\P((\exists n)\, Z_{n}=0)$. Then $q=1$ iff $\E(L)=\sum_{k\ge 0}kp_{k} \le 1$. Moreover, $q$ is the smallest fixed point of $f(s)=\sum_{k\ge 0}p_{k}s^{k}$.
\end{thm}

We are interested in the case where each person has $n$ children, each with probability $p$ of surviving; then the probability $p_{k}$ of $k$ children surviving satisfies
\[
	p_{k}={n\choose k}p^{k}(1-p)^{n-k}
\]
and $\E(L)=np$. In particular, if $n=\Delta=3$ and $p=1/2$ then $q<1$, i.e., there is a positive probability of non-extinction of the family.

A synonym for \emph{Martin-L\"of random} is \emph{1-random}. A more restrictive notion of randomness is \emph{Schnorr 2-randomness} (see Nies' monograph \cite{NiesBook}).
 
\begin{thm}[Law of Computationally Weak Subsets]\label{law}
For each Schnorr 2-random set $R$ there is an infinite set $S\subseteq R$ such that for all $Z\le_{T}S$, $Z$ is not 1-random.
\end{thm}
\begin{proof}
Let $R\subseteq\omega$ be Schnorr 2-random, and let $X\subseteq\omega^{<\omega}$ be the image of $R$ under an effective bijection $h:\omega\to\omega^{<\omega}$. In this situation we say that $X$ is a Schnorr 2-random subset of $\omega^{<\omega}$. Since $h$ induces a map $\hat h:\{R:R\subseteq\omega\}\to\{X:X\subseteq\omega^{<\omega}\}$ given by $\hat h(R)=\{h(n):n\in R\}$, that preserves subsets and infinitude, it suffices to show that there is an infinite set $Y\subseteq X$ such that for all $Z\le_{T}Y$, $Z$ is not 1-random.

By Theorem \ref{Main}, let $C$ be a 3-perfectly bushy subset of $\omega^{<\omega}$, $C\le_{T}0'$, such that for each $W\in [C]$ and each 1-random set $Z$, $Z\not\le_{T} W$. 

Recall that $\sigma\subseteq \tau$ means that $\sigma$ is a substring of $\tau$. Let\footnote{$G_{X}$ can be thought of as a Galton-Watson family tree.} 
\[
	G_{X}=\{\sigma\in\omega^{<\omega} : (\forall\tau\subseteq\sigma)( \tau\in C\cap X)\}.
\]
To connect with the Extinction Criterion \ref{extinctionCriterion}, first write $\{\sigma\in G_{X} : |\sigma|=n\}=\{\sigma^{(n)}_{0},\ldots,\sigma^{(n)}_{Z_{n}}\}$, where for each $0\le t<Z_{n}$, $\sigma^{(n)}_{t}$ precedes $\sigma^{(n)}_{t+1}$ in some fixed computable linear order (say, the lexicographical order). Then for $i\le Z_{n}$, let $L^{(n)}_{i}$ be the cardinality of $\{k: (\sigma^{(n)}_{i})* k\in C\cap X\}$. 

Note that if we consider $X$ as the value of a fair-coin random variable on the power set of $\omega^{<\omega}$, then $L^{(n)}_{i}$ is a binomial random variable with parameters $p=1/2$ and $n=3$. That is, we have a birth-death process where everyone has 3 children, each with a 50$\%$ chance of surviving and 
themselves having 3 children.

Since the branching rate of $C$ is exactly 3, we have a kind of $C$-effective compactness making the event of extinction,
\[
	\{X:[G_{X}]=\nil\}
	=\{X: (\exists n)(\forall \sigma\in\omega^{n})(\sigma\not\in G_{X})\},
\]
into a $\Sigma^0_1(C)$ class. We produce independent copies of this class by letting $X^n=\{\sigma: 0^{n}*1*\sigma\in X\}$ and
\[
	\mc E_{n}=\{X: [G_{X^{n}}]=\nil\}.
\]
Then $\mc E_{n}$ is $\Sigma^{0}_{1}(C)$, the events $\mc E_{n}$, $n\in\omega$, are mutually independent, and $\P(\mc E_{n})=\P([G_X]=\nil)$ for each $n$.
By Theorem \ref{extinctionCriterion}, $q:=\P((\exists n) Z_{n}=0)=\P([G_X]=\nil)$ is the smallest positive fixed point of $f(s)=\sum_{k\ge 0}p_{k}s^{k}=\frac{1}{8}+\frac{3}{8}s+\frac{3}{8}s^{2}+\frac{1}{8}s^{3}$. We find that the equation $f(s)=s$ has its smallest positive solution at $s=\sqrt{5}-2$. So 
\[
	\P(\cap_{k< n}\mc E_{k})=(\mc P(\mc E_{0}))^{n}=(\sqrt{5}-2)^{n},
\]
which is  computable and converges to 0 effectively. Since $X$ is Schnorr random relative to $C$, we have $X\not\in\cap_{n}\mc E_{n}$. So fix $n$ with $X\not\in\mc E_{n}$. Then $[G_{X^{n}}]\ne\nil$, so fix $W\in [G_{X^{n}}]\subseteq [C]$. That is, if $\tau$ is a prefix of $W\in\omega^{\omega}$ then $\tau\in C$ and $0^{n}*\tau\in X$. Since $W\in [C]$, for each 1-random $Z$ we have $Z\not\le_{T}W$.

Let $Y=\{0^{n}*\tau: \tau\text{ is a prefix of }W\}\subseteq X$. Then $Y$ is clearly infinite, and Turing equivalent to $W$, hence $Y$ does not compute any 1-random set.
\end{proof}

\begin{cor}\label{Law}
There is an almost sure event $\mc A$ such that if $X\in\mc A$ then $X$ has an infinite subset $Y$ such that no element of $\mc A$ is Turing reducible to $Y$.
\end{cor}
\begin{proof}
Let $	\mc A=\{X\mid X\text{ is Schnorr 2-random}\}$ and apply Theorem \ref{law}.
\end{proof}

It is of interest for the study of Ramsey's theorem in Reverse Mathematics to know how far the Law of Weak Subsets can be effectivized. This subject is discussed in an earlier paper \cite{K:MRL} and studied in detail by Dzhafarov \cite{Dzhafarov}. 

\begin{que}
Does Corollary \ref{Law} hold with $\mc A=\{X \mid X\text{ is 1-random}\}$? That is, does every 1-random set $X$ have an infinite subset $Y\subseteq X$ such that $Y$ does not compute any 1-random set?
\end{que}

\end{document}